\documentclass[12pt,reqno]{amsart}
\usepackage{enumerate, latexsym, amsmath, amsfonts, amssymb, amsthm, color}
\allowdisplaybreaks[4]
\def\pmod #1{\ ({\rm{mod}}\ #1)}

\def\l{\left}
\def\r{\right}
\def\bg{\bigg}
\def\({\bg(}
\def\){\bg)}

\theoremstyle{plain}
\newtheorem{theorem}{Theorem}

\newtheorem{lemma}{Lemma}

\theoremstyle{definition}

\theoremstyle{remark}
\newtheorem{remark}{Remark}

 \vspace{4mm}

\begin{document}

\hbox{}
\medskip

\title
[{Supercongruences involving binomial coefficients}]
{Supercongruences involving products of two binomial coefficients modulo $p^4$}

\author
[Guo-Shuai Mao] {Guo-Shuai Mao}

\address {(Guo-Shuai Mao) Department of Mathematics, Nanjing
University of Information Science and Technology, Nanjing 210044,  People's Republic of China\\
{\tt maogsmath@163.com  } }

\keywords{Congruences; Central binomial coefficients; Harmonic numbers; Bernoulli numbers.
\newline \indent The author was supported by the National Natural Science Foundation of China (Grant No. 12001288).}
\begin{abstract} In this paper, we mainly prove a congruence conjecture of Z.-W. Sun \cite{Sjnt}: Let $p>5$ be a prime. Then
$$
\sum_{k=(p+1)/2}^{p-1}\frac{\binom{2k}k^2}{k16^k}\equiv-\frac{21}2H_{p-1}\pmod{p^4},
$$
where $H_n$ denotes the $n$-th harmonic number.
\end{abstract}
\maketitle
\section{Introduction}
\setcounter{lemma}{0}
\setcounter{theorem}{0}
\setcounter{corollary}{0}
\setcounter{remark}{0}
\setcounter{equation}{0}
\setcounter{conjecture}{0}
The Bernoulli numbers $\{B_{n}\}$ and Bernoulli polynomials $\{B_{n}(x)\}$ are defined by
\begin{gather*}
 B_0=1,\ \ \ \sum_{k=0}^{n-1}\binom{n}{k}B_{k}=0\ \ (n\geq2),\\ B_n(x)=\sum_{k=0}^n\binom nkB_kx^{n-k}\ \ (n=0,1,2,\ldots).
 \end{gather*}
Let $m>0$ and let $(a_1,a_2,\ldots,a_m)\in{(\mathbb{N})^m}=\substack{\underbrace{\mathbb{N}\times\mathbb{N}\cdots\times\mathbb{N}}\\ m\ times}$, where $\mathbb{N}=\{0,1,2,\ldots\}$. For any $n\geq m$, we define the alternating multiple harmonic sum as
$$
H(a_1,a_2,\ldots,a_m;n)=\sum_{\substack{1\leq k_1<k_2<\ldots<k_m\leq n}}\prod_{i=1}^m\frac{\mbox{sign}(a_i)^{k_i}}{k_i^{|a_i|}}.
$$
The integers $m$ and $\sum_{i=1}^m|a_i|$ are respectively the depth and the weight of the harmonic sum. As a matter of convenience, we remember $H(1;n)$ as $H_n$.
We know several non-alternating harmonic sums modulo a power of a prime as follows:

\noindent{\rm (i)}. (\cite{H}) for $a,r>0$ and for any prime $p>ar+2$
\begin{align*}
H(\{a\}^r;p-1)\equiv\begin{cases}(-1)^r\frac{a(ar+1)}{2(ar+2)}p^2B_{p-ar-2}\ &\pmod{p^3}\qquad \text{if $ar$ is odd},\\
(-1)^{r-1}\frac{a}{ar+1}p B_{p-ar-1}\ &\pmod{p^2}\qquad \text{if $ar$ is even}.\end{cases}
\end{align*}
{\rm(ii)}. (\cite{s2000}) for a positive integer $a$ and for any prime $p>a+2$, we have
\begin{align*}
H\l(a;\frac{p-1}2\r)\equiv\begin{cases}-2q_p(2)\pmod{p} &\text{a=1},\\
-\frac{2^a-2}{a}B_{p-a} \pmod{p} &\text{if $a>1$ is odd},\\
\frac{a(2^{a+1}-1)}{2(a+1)}pB_{p-a-1} \pmod{p^2} &\text{if $a$ is even},\end{cases}
\end{align*}
where $q_p(a)=(a^{p-1}-1)/p$ stands for the Fermat quotient.

\noindent{\rm (iii)}. (\cite{H}) For $a,b>0$ with $a+b$ odd and for any prime $p>a+b+1$, we have
$$
H(a,b;p-1)\equiv\frac{(-1)^b}{a+b}\binom{a+b}aB_{p-a-b}\pmod p.
$$
{\rm (iv)}. (\cite[Lemma 1]{HHT}) if $a,b$ are positive integers and $a+b$ is odd, then for any prime $p>a+b$,
$$
H\l(a,b;\frac{p-1}2\r)\equiv\frac{B_{p-a-b}}{2(a+b)}\l((-1)^b\binom{a+b}a+2^{a+b}-2\r)\pmod p.
$$
{\rm (v)}. (\cite[Corollary 2.3]{TZ}) Let $a\in\mathbb{Z}_{\geq0}$ and $p\geq a+2$ be a prime. Then
\begin{align*}
H(-a;p-1)\equiv\begin{cases}-\frac{2(1-2^{p-a}))}{a}B_{p-a}\ &\pmod{p}\qquad \text{if $a$ is odd},\\
\frac{a(1-2^{p-1-a})}{a+1}pB_{p-1-a}\ &\pmod{p^2}\qquad \text{if $a$ is even}.\end{cases}
\end{align*}
{\rm (vi)}. (\cite[Theorem 3.1]{TZ}) Let $a,b\in\mathbb{N}$ and $p\geq a+b+2$ be a prime. If $a+b$ is odd then we have
$$
H(-a,b;p-1)\equiv H(a,-b;p-1)\equiv\frac{1-2^{p-a-b}}{a+b}B_{p-a-b}\pmod p.
$$
{\rm (vii)}. (\cite[Propositions 6.3 and 7.3, (116)]{TZ}) Let $p>5$ be a prime and 
$$X:=\frac{B_{p-3}}{p-3}-\frac{B_{2p-4}}{4p-8}.$$ Then 
$$
H(-4;p-1)\equiv H(2,2;p-1)\equiv H(1,3;p-1)\equiv0\pmod p,
$$
$$
H(2,-1;p-1)\equiv-\frac32X-\frac76pq_p(2)B_{p-3}+pH(1,-3;p-1)\pmod {p^2},
$$
\begin{align*}
-\frac12H(1,2;p-1)&\equiv\frac12H(2,1;p-1)\equiv H(-3;p-1)\\
&\equiv-2H(1,-2;p-1)\equiv3X\pmod{p^2}
\end{align*}
{\rm (viii)}. (\cite[Theorems 5.1 and 5.2, Remark 5.1]{s2000}) Let $p>3$ be a prime. Then
$$
-\frac2pH_{p-1}\equiv H(2;p-1)\equiv \frac27H\left(2;\frac{p-1}2\right)\equiv-4pX\pmod{p^3},
$$
$$
H\left(3;\frac{p-1}2\right)\equiv12X\pmod{p^2}.
$$
Throughout the paper, $p$ always stands for a prime, and $\mathbb{Z}_p$ denotes the set of $p$-adic integers, and for $a\in\mathbb{Z}_p$, let $\langle a\rangle_p\in\{0,1,\ldots,p-1\}$ be given by $\langle a\rangle_p\equiv a\pmod p$.

\noindent In \cite{Tk}, Tauraso proved that for any prime $p>5$,
\begin{align}\label{T}
\sum_{k=1}^{p-1}\frac{\binom{2k}k^2}{k6^k}\equiv-2H_{\frac{p-1}2}\pmod{p^3},
\end{align}
and in \cite{Sjnt}, Sun proved that
$$
\sum_{k=\frac{p+1}2}^{p-1}\frac{\binom{2k}k^2}{k6^k}\equiv\frac{7}2p^2B_{p-3}\pmod{p^3}.
$$
Tauraso \cite{Tk} also obtained the following congruence: Let $p>5$ be a prime, $a\in\mathbb{Z}_p$ and $t=(a-\langle a\rangle_p)/p$. Then
$$
\sum_{k=1}^{p-1}\frac1k\binom{a}k\binom{-a-1}k\equiv-2H_{\langle a\rangle_p}+2ptH(2;\langle a\rangle_p)\pmod{p^2}.
$$ 
Sun \cite{s2015} generalised Tauraso's result to
\begin{align*}
&\sum_{k=1}^{p-1}\frac1k\binom{a}k\binom{-a-1}k\\
&\equiv-\frac23p^2tB_{p-3}-2H_{\langle a\rangle_p}+2ptH(2;\langle a\rangle_p)+2p^2tH(3;\langle a\rangle_p)\pmod{p^3}.
\end{align*}
Motivated by the above, we generalised Z.-H. Sun's result and confirm a conjecture of Z.-W. Sun \cite{S11a}:
\begin{theorem}\label{Th1.1} Let $p>5$ be a prime, $a\in\mathbb{Z}_p$ and $t:=(a-\langle a\rangle_p)/p$. Then
\begin{align*}
&\sum_{k=1}^{p-1}\frac1k\binom{a}k\binom{-1-a}k\equiv4p^2tX-2H_{\langle a\rangle_p}+2ptH(2;\langle a\rangle_p)+2p^2tH(3;\langle a\rangle_p)\notag\\
&-2p^3t(2t^2+4t+1)H(4;\langle a\rangle_p)+4p^3t(t+1)\sum_{k=1}^{\langle a\rangle_p}\frac{H_k}{k^3}\pmod{p^4},
\end{align*}
and if $\langle a\rangle_p\leq (p-1)/2$,
\begin{align*}
&\sum_{k=1}^{\frac{p-1}2}\frac1k\binom{a}k\binom{-1-a}k\\
&\equiv-12p^2t^2X+14p^2tX-2H_{\langle a\rangle_p}+4ptH(2;\langle a\rangle_p)-6p^2t^2H(3;\langle a\rangle_p)\notag\\
&+8p^3t^3H(4;\langle a\rangle_p)+4p^2t\sum_{k=1}^{\langle a\rangle_p}\frac1{k^2}\sum_{j=1}^k\frac1{2j-1}-8p^3t^2\sum_{k=1}^{\langle a\rangle_p}\frac1{k^3}\sum_{j=1}^k\frac1{2j-1}\notag\\
&+4p^3t\sum_{k=1}^{\langle a\rangle_p}\frac1{k^2}\left(\sum_{j=1}^k\frac1{2j-1}\right)^2-8p^3t^2\sum_{k=1}^{\langle a\rangle_p}\frac{1}{k^2}\sum_{j=1}^k\frac1{(2j-1)^2}\pmod{p^4}.
\end{align*}
Furthermore,
\begin{equation}\label{1.0}
\sum_{k=1}^{p-1}\frac{\binom{2k}k^2}{k16^k}\equiv-2H_{\frac{p-1}2}-p^3\sum_{k=1}^{\frac{p-1}2}\frac{H_k}{k^3}\pmod{p^4}.
\end{equation}
\begin{equation}\label{1.1}
\sum_{k=\frac{p+1}2}^{p-1}\frac{\binom{2k}k^2}{k16^k}\equiv-\frac{21}2H_{p-1}\pmod{p^4}.
\end{equation}
\end{theorem}
\begin{remark} (\ref{1.1}) was conjectured by Z.-W. Sun \cite[Conjecture 1.1]{S11a}.
\end{remark}
In 2008, Z.-H. Sun \cite[Remark 4.1]{s2008} proved that for any prime $p>3$,
$$
\sum_{k=1}^{p-1}\frac{2^k}{k^3}\equiv-\frac13q^3_p(2)-\frac7{24}B_{p-3}\pmod p.
$$ 
Tauraso and Zhao \cite[(79)]{TZ} reproved the above congruence in 2010, Mattarei and Tauraso \cite{MT} also reproved this congruence in 2013. And Tauraso and Zhao \cite[Remark 7.7]{TZ} said that they not able to express $H(1,1,-1;p-1)\pmod{p^2}$ explicitly, but they found that it is equivalent to determining $\sum_{k=1}^{p-1}\frac{2^k}{k^3}\pmod{p^2}$.

\noindent In this paper, we give $\sum_{k=1}^{p-1}\frac{2^k}{k^3}\pmod{p^2}$.
\begin{theorem}\label{Th1.2} For any prime $p>5$, we have
$$
\sum_{k=1}^{p-1}\frac{2^k}{k^3}\equiv-\frac13q^3_p(2)+\frac74X+\frac5{12}pq^4_p(2)+\frac{7p}6q_p(2)B_{p-3}-\frac{3p}8\sum_{k=1}^{\frac{p-1}2}\frac{H_k}{k^3}\pmod{p^2}.
$$
\end{theorem}
\noindent We are going to prove Theorem \ref{Th1.1} in Section 2. Section 3 is devoted to proving Theorem \ref{Th1.2}. Our proofs make use of some combinatorial identities which were found by the package \texttt{Sigma} \cite{S} via the software \texttt{Mathematica}. The proof of Theorem \ref{Th1.1} is somewhat difficult and complex because it is rather convoluted. We are not able to know the exact result of 
$$\sum_{k=1}^{\frac{p-1}2}\frac{H_k}{k^3}\equiv H\left(1,3;\frac{p-1}2\right)\equiv4q_p(2)B_{p-3}-H\left(3,1;\frac{p-1}2\right)\pmod p,$$  
so we just keep it in the congruences of above and following.
\section{Proof of Theorem \ref{Th1.1}}
\setcounter{lemma}{0}
\setcounter{theorem}{0}
\setcounter{corollary}{0}
\setcounter{remark}{0}
\setcounter{equation}{0}
\setcounter{conjecture}{0}
\begin{lemma}\label{Lem2.3} Let $p>5$ be a prime, and let $t\in\mathbb{Z}_p$. If $k\in\{1,2,\ldots,p-1\}$, then
\begin{align*}
&\binom{pt+k-1}{p-1}\binom{-pt-k-1}{p-1}\\
&\equiv\frac{p^2t(t+1)}{k^2}\left(1+2pH_k-\frac{p}k-\frac{2pt}k\right)\pmod{p^4}.
\end{align*}
If $k\in\{1,2,\ldots,(p-1)/2\}$, then
\begin{align*}
&\binom{pt+k-1}{\frac{p-1}2}\binom{-pt-k-1}{\frac{p-1}2}\equiv\frac{pt}{k}\bigg(1-\frac{pt}k+2p\sum_{j=1}^k\frac1{2j-1}+\frac{p^2t^2}{k^2}\\
+&2p^2\left(\sum_{j=1}^k\frac1{2j-1}\right)^2-\frac{2p^2t}k\sum_{j=1}^k\frac1{2j-1}-4p^2t\sum_{j=1}^k\frac1{(2j-1)^2}\bigg)\pmod{p^4}.
\end{align*}
\end{lemma}
\begin{proof} It is easy to check that
\begin{align*}
&\binom{pt+k-1}{p-1}=\frac{(pt+k-1)\cdots(pt+1)pt(pt-1)\cdots(pt+k-p+1)}{(p-1)!}\\
&\equiv\frac{pt(k-1)!(1+ptH_{k-1})(-1)^{p-1-k}(p-1-k)!(1-ptH_{p-1-k})}{(p-1)!}\\
&\equiv\frac{pt}k\left(1+pH_k-\frac{pt}k\right)\pmod{p^3},
\end{align*}
and by (i), we have
\begin{align*}
&\binom{-pt-k-1}{p-1}\\
&=\frac{(pt+k+1)\cdots(pt+p-1)p(t+1)(pt+p+1)\cdots(pt+p+k-1)}{(p-1)!}\\
&\equiv\frac{p(t+1)(p-1)!(1+pt(H_{p-1}-H_k))(k-1)!(1+p(t+1)H_{k-1})}{k!(p-1)!}\\
&\equiv\frac{p(t+1)}k\left(1+pH_{k-1}-\frac{pt}k\right)\pmod{p^3},
\end{align*}
hence
$$
\binom{pt+k-1}{p-1}\binom{-pt-k-1}{p-1}\equiv\frac{p^2t(t+1)}{k^2}\left(1+2pH_k-\frac{p}k-\frac{2pt}k\right)\pmod{p^4}.
$$
Similarly,
\begin{align*}
&\binom{pt+k-1}{\frac{p-1}2}=\frac{(pt+k-1)\cdots(pt+1)pt(pt-1)\cdots(pt+k-\frac{p-1}2)}{(\frac{p-1}2)!}\\
&\equiv\frac{pt(k-1)!(1+ptH_{k-1}+p^2t^2H(1,1;k-1))(-1)^{\frac{p-1}2-k}(\frac{p-1}2-k)!}{(\frac{p-1}2)!}\\
&\ \ \ \ \ \ \ \ \ \ \ \times(1-ptH_{\frac{p-1}2-k}+p^2t^2H(1,1;\frac{p-1}2-k))\\
&\equiv\frac{pt(-1)^{\frac{p-1}2-k}}{k\binom{\frac{p-1}2}k}\bigg(1+ptH_{k-1}-ptH_{\frac{p-1}2-k}-p^2t^2H_{k-1}H_{\frac{p-1}2-k}\\
&\ \ \ \ \ \ \ \ \ \ \ \ \  \ +p^2t^2H(1,1;k-1)+p^2t^2H(1,1;\frac{p-1}2-k)\bigg)\pmod{p^4}
\end{align*}
and
\begin{align*}
&\binom{-pt-k-1}{\frac{p-1}2}=\frac{(-1)^{\frac{p-1}2}(pt+k+1)\cdots(pt+k+\frac{p-1}2)}{(\frac{p-1}2)!}\\
&\equiv(-1)^{\frac{p-1}2}\binom{\frac{p-1}2+k}{k}\bigg(1+ptH_{\frac{p-1}2+k}-ptH_k-p^2t^2H_kH_{\frac{p-1}2+k}\\
&+p^2t^2H_k^2-p^2t^2H(1,1;k)+p^2t^2H(1,1;\frac{p-1}2+k)\bigg)\pmod{p^3},
\end{align*}
Hence by the fact that $H_{p-1-k}\equiv H_k\pmod p$, $H(2;p-1-k)\equiv-H(2;k)\pmod p$ and
$$
H(1,1;k)-H(1,1;k-1)=\frac12(H_k^2-H_{k-1}^2-H(2;k)+H(2;k-1))=\frac{H_k}k-\frac1{k^2}
$$
we have
\begin{align*}
&\binom{pt+k-1}{\frac{p-1}2}\binom{-pt-k-1}{\frac{p-1}2}\\
&\equiv\frac{pt(-1)^k\binom{\frac{p-1}2+k}{k}}{k\binom{\frac{p-1}2}k}\left(1-\frac{pt}k+\frac{p^2t^2}{k^2}+ptH_{\frac{p-1}2+k}-ptH_{\frac{p-1}2-k}\right)\pmod{p^4}
\end{align*}
In view of \cite[(3.2)]{Su3} and \cite[Lemma 4.2]{S11a}, we have
$$(-1)^k\binom{\frac{p-1}2+k}{k}\binom{\frac{p-1}2}k\left(1-\frac{p}4(H_{\frac{p-1}2+k}-H_{\frac{p-1}2-k})\right)\equiv\frac{\binom{2k}k^2}{16^k}\pmod{p^4}$$
and
$$
\frac{\binom{\frac{p-1}2}k^216^k}{\binom{2k}k^2}\equiv1-2p\sum_{j=1}^k\frac1{2j-1}+2p^2\left(\sum_{j=1}^k\frac1{2j-1}\right)^2-p^2\sum_{j=1}^k\frac1{(2j-1)^2}\pmod{p^3}.
$$
Thus by the first congruence in \cite[pp 16]{Su3}, we have
\begin{align*}
&\frac{(-1)^k\binom{\frac{p-1}2+k}{k}}{k\binom{\frac{p-1}2}k}=\frac{(-1)^k\binom{\frac{p-1}2+k}{k}\binom{\frac{p-1}2}k}{k\binom{\frac{p-1}2}k^2}\\
&\equiv1+2p\sum_{j=1}^k\frac1{2j-1}+2p^2\left(\sum_{j=1}^k\frac1{2j-1}\right)^2\pmod{p^3}.
\end{align*}
Therefore by the first congruence in \cite[pp 16]{Su3} again, we immediately obtain that
\begin{align*}
&\binom{pt+k-1}{\frac{p-1}2}\binom{-pt-k-1}{\frac{p-1}2}\equiv\frac{pt}{k}\bigg(1-\frac{pt}k+2p\sum_{j=1}^k\frac1{2j-1}+\frac{p^2t^2}{k^2}\\
+&2p^2\left(\sum_{j=1}^k\frac1{2j-1}\right)^2-\frac{2p^2t}k\sum_{j=1}^k\frac1{2j-1}-4p^2t\sum_{j=1}^k\frac1{(2j-1)^2}\bigg)\pmod{p^4}.
\end{align*}
These complete the proof of Lemma \ref{Lem2.3}.
\end{proof}
\noindent{\it Proof of Theorem \ref{Th1.1}}. Set $S_n(a)=\sum_{k=1}^n\frac{1}k\binom{a}k\binom{-1-a}k$, then in view of \cite[(2.1)]{s2015}, we have
$S_n(a)-S_n(a-1)=-\frac2{a}+\frac2a\binom{a-1}{n}\binom{-a-1}n.$ Thus, by Lemma \ref{Lem2.3}, we have
\begin{align}\label{p-1S}
&S_{p-1}(a)-S_{p-1}(a-\langle a\rangle_p)=\sum_{k=0}^{\langle a\rangle_p-1}(S_{p-1}(a-k)-S_{p-1}(a-k-1))\notag\\
&=2\sum_{k=0}^{\langle a\rangle_p-1}\left(\frac{-1}{a-k}+\frac1{a-k}\binom{a-k-1}{p-1}\binom{-a+k-1}{p-1}\right)\notag\\
&\equiv2\sum_{k=1}^{\langle a\rangle_p}\left(\frac{-1}{pt+k}+\frac{1}{pt+k}\binom{pt+k-1}{p-1}\binom{-pt-k-1}{p-1}\right)\notag\\
&\equiv2\sum_{k=1}^{\langle a\rangle_p}\frac{-1}{(pt+k)^3}+2p^2t(t+1)\bigg(H(3;\langle a\rangle_p)-(3pt+p)H(4;\langle a\rangle_p)\notag\\
&\ \ \ \ \ \ \ \ \ \ \ \ \ \ \ \ \ \ \ \ \ \ \ \ \ \ \ \ \ \ \ \ \ \ \ \ \ \ \ \ \ \ +2p\sum_{k=1}^{\langle a\rangle_p}\frac{H_k}{k^3}\bigg)\pmod{p^4}.
\end{align}
Similarly, by Lemma \ref{Lem2.3}, we have
\begin{align}\label{p-12S}
&S_{\frac{p-1}2}(a)-S_{\frac{p-1}2}(a-\langle a\rangle_p)=\sum_{k=0}^{\langle a\rangle_p}(S_{\frac{p-1}2}(a-k)-S_{\frac{p-1}2}(a-k-1))\notag\\
&=2\sum_{k=0}^{\langle a\rangle_p-1}\left(\frac{-1}{a-k}+\frac1{a-k}\binom{a-k-1}{\frac{p-1}2}\binom{-a+k-1}{\frac{p-1}2}\right)\notag\\
&=2\sum_{k=1}^{\langle a\rangle_p}\left(\frac{-1}{pt+k}+\frac{1}{pt+k}\binom{pt+k-1}{\frac{p-1}2}\binom{-pt-k-1}{\frac{p-1}2}\right)\notag\\
&\equiv2\sum_{k=1}^{\langle a\rangle_p}\frac{-1}{pt+k}+2ptH(2;\langle a\rangle_p)-4p^2t^2H(3;\langle a\rangle_p)+6p^3t^3H(4;\langle a\rangle_p)\notag\\
&+4p^2t\sum_{k=1}^{\langle a\rangle_p}\frac1{k^2}\sum_{j=1}^k\frac1{2j-1}-8p^3t^2\sum_{k=1}^{\langle a\rangle_p}\frac1{k^3}\sum_{j=1}^k\frac1{2j-1}\notag\\
&+4p^3t\sum_{k=1}^{\langle a\rangle_p}\frac1{k^2}\left(\sum_{j=1}^k\frac1{2j-1}\right)^2-8p^3t^2\sum_{k=1}^{\langle a\rangle_p}\frac{1}{k^2}\sum_{j=1}^k\frac1{(2j-1)^2}\pmod{p^4}.
\end{align}
\begin{lemma}\label{Lem2.4} Let $p>5$ be a prime, $t=\frac{a-\langle a\rangle_p}p\in\mathbb{Z}_p$ and set $X$ as above in (vii). Then
\begin{align*}
S_{p-1}(a-\langle a\rangle_p)=S_{p-1}(pt)\equiv4p^2tX\pmod{p^4},
\end{align*}
\begin{align*}
S_{\frac{p-1}2}(a-\langle a\rangle_p)=S_{\frac{p-1}2}(pt)\equiv-12p^2t^2X+14p^2tX\pmod{p^4}.
\end{align*}
\end{lemma}
\begin{proof} It is easy to see that
\begin{align*}
&\binom{pt}k\binom{-1-pt}k=\frac{pt(pt+k)(-1)^k}{k!^2}\prod_{j=1}^{k-1}(p^2t^2-j^2)\\
&\equiv-\frac{pt}k\left(1+\frac{pt}k\right)(1-p^2t^2H(2;k-1))\pmod{p^4}.
\end{align*}
Thus, by (i), (vii) and (viii), we have
\begin{align*}
&S_{p-1}(pt)=\sum_{k=1}^{p-1}\frac{1}{k}\binom{pt}k\binom{-1-pt}k\\
&\equiv-ptH(2;p-1)-p^2t^2H(3;p-1)+p^3t^3H(2,2;p-1)\\
&\equiv4p^2tX\pmod{p^4}.
\end{align*}
Similarly, by the identity $2H(2,2;\frac{p-1}2)=H(2;\frac{p-1}2)^2-H(4;\frac{p-1}2)$, (ii) and (viii), we have
\begin{align*}
&S_{\frac{p-1}2}(pt)\\
&\equiv-ptH\left(2;\frac{p-1}2\right)-p^2t^2H\left(3;\frac{p-1}2\right)+p^3t^3H\left(2,2;\frac{p-1}2\right)\\
&\equiv-12p^2t^2X+14p^2tX\pmod{p^4}.
\end{align*}
These prove Lemma \ref{Lem2.4}.
\end{proof}
\begin{lemma}\label{Lem2.5} For all prime $p>5$ , $h_{31}:=H(3,1,\frac{p-1}2)$, and set $X$ as above. Then
\begin{align}
&\sum_{k=1}^{\frac{p-1}2}\frac{1}{k^2}\sum_{j=1}^k\frac1{2j-1}\equiv-\frac{21}2X+2pq_p(2)B_{p-3}-\frac{p}2h_{31}\pmod {p^2},\label{ds1}\\
&\sum_{k=1}^{\frac{p-1}2}\frac{H_k}{k^3}\equiv4q_p(2)B_{p-3}-h_{31}\pmod p,\label{ds2}\\
&\sum_{k=1}^{\frac{p-1}2}\frac1k\sum_{j=1}^k\frac1{(2j-1)^2}\notag\\
&\equiv\frac{21}4X-pq_p(2)B_{p-3}+pH(1,-3;p-1)+\frac{p}4h_{31}\pmod{p^2}.\label{ds3}
\end{align}
\end{lemma}
\begin{proof} It is easy to see that
\begin{align*}
&\sum_{k=1}^{\frac{p-1}2}\frac{H_{2k}}{k^2}=2\sum_{k=1}^{p-1}\frac{(1+(-1)^k)H_k}{k^2}\\
=&2(H(1,2;p-1)+H(3;p-1)+H(1,-2;p-1)+H(-3;p-1)).
\end{align*}
So in view of (i) and (vii), we have
$$
\sum_{k=1}^{\frac{p-1}2}\frac{H_{2k}}{k^2}\equiv-9X\pmod {p^2}.
$$
It is easy to see that
\begin{align*}
&H\left(1,2;\frac{p-1}2\right)+H\left(3;\frac{p-1}2\right)=\sum_{k=1}^{\frac{p-1}2}\frac{H_{k}}{k^2}=\sum_{k=1}^{\frac{p-1}2}\frac{1}{k^2}\sum_{j=1}^k\frac1j\\
&=\sum_{j=1}^{\frac{p-1}2}\frac1j\sum_{k=j}^{\frac{p-1}2}\frac1{k^2}=H_{\frac{p-1}2}H\left(2;\frac{p-1}2\right)-H\left(2,1;\frac{p-1}2\right).
\end{align*}
This with (ii) and (viii) yields that
\begin{align*}
&H\left(1,2;\frac{p-1}2\right)+H\left(2,1;\frac{p-1}2\right)\\
&=H_{\frac{p-1}2}H\left(2;\frac{p-1}2\right)-H\left(3;\frac{p-1}2\right)\\
&\equiv-\frac{14}3pq_p(2)B_{p-3}-12X\pmod{p^2}.
\end{align*}
In view of \cite[(112)-(115)]{TZ}, we have
\begin{align*}
&H\left(1,2;\frac{p-1}2\right)-H\left(2,1;\frac{p-1}2\right)\\
&\equiv-6X-\frac{10}3pq_p(2)B_{p-3}+2ph_{31}\pmod{p^2}.
\end{align*}
Thus,
$$
H\left(1,2;\frac{p-1}2\right)\equiv-9X-4pq_p(2)B_{p-3}+ph_{31}\pmod{p^2}
$$
and
\begin{align}\label{21p-12}
H\left(2,1;\frac{p-1}2\right)\equiv-3X-\frac23pq_p(2)B_{p-3}-ph_{31}\pmod{p^2}.
\end{align}
So by (viii), we have
\begin{align*}
\sum_{k=1}^{\frac{p-1}2}\frac{H_{k}}{k^2}&=H\left(1,2;\frac{p-1}2\right)+H\left(3;\frac{p-1}2\right)\\
&\equiv3X-4pq_p(2)B_{p-3}+ph_{31}\pmod{p^2}.
\end{align*}
Hence
\begin{align*}
\sum_{k=1}^{\frac{p-1}2}\frac{1}{k^2}\sum_{j=1}^k\frac1{2j-1}&=\sum_{k=1}^{\frac{p-1}2}\frac{H_{2k}-\frac12H_k}{k^2}\\
&\equiv-\frac{21}2X+2pq_p(2)B_{p-3}-\frac{p}2h_{31}\pmod{p^2}.
\end{align*}
Similarly,
\begin{align*}
\sum_{k=1}^{\frac{p-1}2}\frac{H_k}{k^3}&=H_{\frac{p-1}2}H\left(3;\frac{p-1}2\right)-H\left(3,1;\frac{p-1}2\right)\\
&\equiv4q_p(2)B_{p-3}-h_{31}\pmod p
\end{align*}
and 
\begin{align*}
&\sum_{k=1}^{\frac{p-1}2}\frac1k\sum_{j=1}^k\frac1{(2j-1)^2}=\sum_{k=1}^{p-1}\frac{1+(-1)^k}kH(2;k)-\frac14\sum_{k=1}^{\frac{p-1}2}\frac{H(2;k)}k\\
&=H(2,1;p-1)+H(3;p-1)+H(2,-1;p-1)+H(-3;p-1)\\
&\ \ \ \ \ -\frac14H\left(2,1;\frac{p-1}2\right)-\frac14H\left(3;\frac{p-1}2\right).
\end{align*}
By (\ref{21p-12}), (i), (vii) and (viii), we have
\begin{align*}
&\sum_{k=1}^{\frac{p-1}2}\frac1k\sum_{j=1}^k\frac1{(2j-1)^2}\notag\\
&\equiv\frac{21}4X-pq_p(2)B_{p-3}+pH(1,-3;p-1)+\frac{p}4h_{31}\pmod{p^2}.
\end{align*}
These complete the proof of Lemma \ref{Lem2.5}.
\end{proof}
\begin{lemma}\label{Lemzhongyao} Let $p>5$ be a prime. Then
$$
H\left(1,3;\frac{p-1}2\right)\equiv4H(1,-3;p-1)\pmod p.
$$
\end{lemma}
\begin{proof} In view of \cite[p. 804, 23.1.4-7]{AS} or \cite[(7)]{TZ}, and by Fermat's Little Theorem, we have
\begin{align*}
&H\left(1,3;\frac{p-1}2\right)=\sum_{k=1}^{\frac{p-1}2}\frac1{k^3}\sum_{j=1}^{k-1}\frac1j\equiv\sum_{k=1}^{\frac{p-1}2}\frac1{k^3}\sum_{j=1}^{k-1}j^{p-2}\\
&=\sum_{r=0}^{p-2}\binom{p-1}r\frac{B_r}{p-1}\sum_{k=1}^{\frac{p-1}2}k^{p-4-r}=\sum_{r=0}^{p-6}\binom{p-1}r\frac{B_r}{p-1}\sum_{k=1}^{\frac{p-1}2}k^{p-4-r}\\
&\ \ \ \ \ \ \ \ \ \ \ \ \ \ \ \ \ \ \ \ \ \ \ \ +\sum_{r=p-5}^{p-2}\binom{p-1}r\frac{B_r}{p-1}\sum_{k=1}^{\frac{p-1}2}k^{p-4-r}\pmod p.
\end{align*}
By Fermat's Little Theorem again, (ii) and $B_n=0$ for all odd $n>1$, we have
\begin{align*}
&\sum_{r=0}^{p-6}\binom{p-1}r\frac{B_r}{p-1}\sum_{k=1}^{\frac{p-1}2}k^{p-4-r}\equiv\sum_{r=0}^{p-6}(-1)^{r+1}B_r\sum_{k=1}^{\frac{p-1}2}\frac1{r+3}\\
&\equiv\sum_{r=0}^{p-6}(-1)^rB_r\frac{2^{r+3}-2}{r+3}B_{p-3-r}=\sum_{r=0}^{p-6}B_r\frac{2^{r+3}-2}{r+3}B_{p-3-r}\pmod{p},
\end{align*}
and since $p-2>1$ and $p-4>1$ are odd,
\begin{align*}
&\sum_{r=p-5}^{p-2}\binom{p-1}r\frac{B_r}{p-1}\sum_{k=1}^{\frac{p-1}2}k^{p-4-r}\\
&=\binom{p-1}{p-3}\frac{B_{p-3}}{p-1}H_{\frac{p-1}2}+\binom{p-1}{p-5}\frac{B_{p-5}}{p-1}\sum_{k=1}^{\frac{p-1}2}k\\
&\equiv2q_p(2)B_{p-3}+\frac18B_{p-5}\pmod p.
\end{align*}
Thus, with $B_2=1/6$, we have
\begin{align*}
&H\left(1,3;\frac{p-1}2\right)\equiv\sum_{r=0}^{p-6}B_r\frac{2^{r+3}-2}{r+3}B_{p-3-r}+2q_p(2)B_{p-3}+\frac18B_{p-5}\\
&\equiv\sum_{r=0}^{p-5}B_r\frac{2^{r+3}-2}{r+3}B_{p-3-r}+2q_p(2)B_{p-3}\\
&\equiv2\sum_{r=2}^{p-3}\frac{1-2^{p-1-r}}rB_rB_{p-3-r}+2q_p(2)B_{p-3}\pmod p.
\end{align*}
In view of \cite[pp. 14]{TZ}, we have the following congruences modulo $p$,
\begin{align*}
&H(-1,3;p-1)\equiv2\sum_{r=2}^{p-3}(1-2^r)(1-2^{p-3-r})\frac{B_rB_{p-3-r}}r-2q_p(2)B_{p-3}\\
\equiv&2\sum_{r=2}^{p-3}\frac{1-2^r}rB_rB_{p-3-r}+\frac12\sum_{r=2}^{p-3}\frac{1-2^{p-1-r}}rB_rB_{p-3-r}-2q_p(2)B_{p-3}
\end{align*}
and 
$$
H(1,-3;p-1)\equiv-2\sum_{r=2}^{p-3}\frac{1-2^r}rB_rB_{p-3-r}+2q_p(2)B_{p-3}.
$$
These with \cite[(55)]{TZ} yield that
\begin{align*}
H\left(1,3;\frac{p-1}2\right)&\equiv4\left(H(-1,3;p-1)+H(1,-3;p-1)\right)+2q_p(2)B_{p-3}\\
&\equiv4H(1,-3;p-1)\pmod p.
\end{align*}
\noindent This proves Lemma \ref{Lemzhongyao}.
\end{proof}
\begin{lemma}\label{Lem2.6} For any prime $p>5$, we have the following modulo $p$
$$
\sum_{k=1}^{\frac{p-1}2}\frac1{k^2}\left(\sum_{j=1}^k\frac1{2j-1}\right)^2+\sum_{k=1}^{\frac{p-1}2}\frac1{k^3}\sum_{j=1}^k\frac1{2j-1}+\sum_{k=1}^{\frac{p-1}2}\frac1{k^2}\sum_{j=1}^k\frac1{(2j-1)^2}\equiv0.
$$
\end{lemma}
\begin{proof} We can find and prove the following identity by \texttt{Sigma} \cite{S},
$$
\sum_{k=1}^n\frac{\binom{n}k(-4)^k}{k^2\binom{2k}k}\sum_{j=1}^k\frac1{2j-1}=-2\sum_{k=1}^n\frac1k\sum_{j=1}^k\frac1{(2j-1)^2}.
$$
Set $n=(p-1)/2$ in the above identity and by \cite[Lemma 4.2]{S11a}, we have
\begin{align*}
&-2\sum_{k=1}^{\frac{p-1}2}\frac1k\sum_{j=1}^k\frac1{(2j-1)^2}=\sum_{k=1}^{\frac{p-1}2}\frac{\binom{\frac{p-1}2}k(-4)^k}{k^2\binom{2k}k}\sum_{j=1}^k\frac1{2j-1}\\
&\equiv\sum_{k=1}^{\frac{p-1}2}\frac1{k^2}\bigg(1-p\sum_{j=1}^k\frac1{2j-1}\bigg)\sum_{j=1}^k\frac1{2j-1}\\
&=\sum_{k=1}^{\frac{p-1}2}\frac1{k^2}\sum_{j=1}^k\frac1{2j-1}-p\sum_{k=1}^{\frac{p-1}2}\frac1{k^2}\left(\sum_{j=1}^k\frac1{2j-1}\right)^2\pmod{p^2}.
\end{align*}
This, with (\ref{ds1}) and (\ref{ds3}) yields that
\begin{align}\label{u1}
&\sum_{k=1}^{\frac{p-1}2}\frac1{k^2}\left(\sum_{j=1}^k\frac1{2j-1}\right)^2\equiv2H(1,-3;p-1)\pmod{p}.
\end{align}
It is easy to see that
\begin{align*}
&\sum_{k=1}^{\frac{p-1}2}\frac1{k^2}\sum_{j=1}^k\frac1{(2j-1)^2}=2H(2;2;p-1)+2H(4;p-1)+2H(2,-2;p-1)\\
&\ \ \ \ \ \ \ \ \ \ \ \ \ \ \ \ \ \ \ \ \ \ \ +2H(-4;p-1)-\frac14H\left(2;2;\frac{p-1}2\right)-\frac14H\left(4;\frac{p-1}2\right).
\end{align*}
So in view of (i), (ii), (vii), \cite[(60]{TZ} and 
$$
H\left(2;2;\frac{p-1}2\right)=\frac12\left(H\left(2;\frac{p-1}2\right)^2-H\left(4;\frac{p-1}2\right)\right)\equiv0\pmod p,
$$ 
we have
\begin{align}\label{u2}
\sum_{k=1}^{\frac{p-1}2}\frac1{k^2}\sum_{j=1}^k\frac1{(2j-1)^2}\equiv-2H(-2,2;p-1)\pmod p.
\end{align}
Similarly, by (i), (ii) and (vii), we have
\begin{align}\label{u3}
&\sum_{k=1}^{\frac{p-1}2}\frac1{k^3}\sum_{j=1}^k\frac1{2j-1}=4H(1,3;p-1)+4H(4;p-1)+4H(1,-3;p-1)\notag\\
&\ \ \ \ \ \ \ \ \ \ \ \ \ \ \ \ \ \ \  +4H(-4;p-1)-\frac12H\left(1,3;\frac{p-1}2\right)-\frac12H\left(4;\frac{p-1}2\right)\notag\\
&\equiv 4H(1,-3;p-1)-\frac12H\left(1,3;\frac{p-1}2\right)\pmod p.
\end{align}
These, with Lemma \ref{Lemzhongyao}, \cite[(61) and (62)]{TZ} yield that
\begin{align*}
&\sum_{k=1}^{\frac{p-1}2}\frac1{k^2}\left(\sum_{j=1}^k\frac1{2j-1}\right)^2+\sum_{k=1}^{\frac{p-1}2}\frac1{k^2}\sum_{j=1}^k\frac1{(2j-1)^2}+\sum_{k=1}^{\frac{p-1}2}\frac1{k^3}\sum_{j=1}^k\frac1{2j-1}\\
\equiv&2H(1,-3;p-1)-2H(-2,2;p-1)+4H(1,-3;p-1)-\frac12H\left(1,3;\frac{p-1}2\right)\\
\equiv&-2H(-2,2;p-1)+4H(1,-3;p-1)\equiv0\pmod p.
\end{align*}
\noindent Now the proof of Lemma \ref{Lem2.6} is complete.
\end{proof}
\noindent In view of (\ref{p-1S}), (\ref{p-12S}) and Lemma \ref{Lem2.4}, we have
\begin{align*}
&S_{p-1}(a)\equiv4p^2tX+2\sum_{k=1}^{\langle a\rangle_p}\frac{-1}{pt+k}+2p^2t(t+1)\bigg(H(3;\langle a\rangle_p)\notag\\
&\ \ \ \ \ \ \ \ \ \ \ \ \ \ \ \ \ \ \ \ \ \ \ \ \ \ \ \ -(3pt+p)H(4;\langle a\rangle_p)+2p\sum_{k=1}^{\langle a\rangle_p}\frac{H_k}{k^3}\bigg)\pmod{p^4}.
\end{align*}
and
\begin{align*}
&S_{\frac{p-1}2}(a)\\
&\equiv-12p^2t^2X+14p^2tX+2\sum_{k=1}^{\langle a\rangle_p}\frac{-1}{pt+k}+2ptH(2;\langle a\rangle_p)-4p^2t^2H(3;\langle a\rangle_p)\notag\\
&+6p^3t^3H(4;\langle a\rangle_p)+4p^2t\sum_{k=1}^{\langle a\rangle_p}\frac1{k^2}\sum_{j=1}^k\frac1{2j-1}-8p^3t^2\sum_{k=1}^{\langle a\rangle_p}\frac1{k^3}\sum_{j=1}^k\frac1{2j-1}\notag\\
&+4p^3t\sum_{k=1}^{\langle a\rangle_p}\frac1{k^2}\left(\sum_{j=1}^k\frac1{2j-1}\right)^2-8p^3t^2\sum_{k=1}^{\langle a\rangle_p}\frac{1}{k^2}\sum_{j=1}^k\frac1{(2j-1)^2}\pmod{p^4}.
\end{align*}
Set $a=-1/2$, then $t=-1/2$, thus by (ii), (viii), Lemmas \ref{Lem2.5} and \ref{Lem2.6}, we have
$$
\sum_{k=1}^{p-1}\frac{\binom{2k}k^2}{k16^k}=S_{p-1}\left(-\frac12\right)\equiv-2H_{\frac{p-1}2}-p^3\sum_{k=1}^{\frac{p-1}2}\frac{H_k}{k^3}\pmod{p^4}
$$
and
\begin{align*}
&\sum_{k=\frac{p+1}2}^{p-1}\frac{\binom{2k}k^2}{k16^k}=S_{p-1}\left(-\frac12\right)-S_{\frac{p-1}2}\left(-\frac12\right)\\
&\equiv-21p^2X+2p^3\sum_{k=1}^{\frac{p-1}2}\frac1{k^2}\left(\sum_{j=1}^k\frac1{2j-1}\right)^2+2p^3\sum_{k=1}^{\frac{p-1}2}\frac{1}{k^2}\sum_{j=1}^k\frac1{(2j-1)^2}\\
&\ \ \ \ \ \ \ \ \ \ \ \ \ \ \ \ \ \ \ \ \ \ \ \ \ \ \ +2p^3\sum_{k=1}^{\frac{p-1}2}\frac1{k^3}\sum_{j=1}^k\frac1{2j-1}\\
&\equiv-21p^2X\equiv-\frac{21}2H_{p-1}\pmod{p^4}.
\end{align*}
\noindent Now the proof of Theorem \ref{Th1.1} is complete.\qed
\section{Proof of Theorem \ref{Th1.2}}
\setcounter{lemma}{0}
\setcounter{theorem}{0}
\setcounter{corollary}{0}
\setcounter{remark}{0}
\setcounter{equation}{0}
\setcounter{conjecture}{0}
\begin{lemma}\label{Lem3.1} For any prime $p>5$, we have
\begin{align*}
H_{\frac{p-1}2}\equiv-2q_p(2)+pq^2_p(2)-\frac23p^2q^3_p(2)+\frac12p^3q^4_p(2)+\frac72p^2X\pmod{p^4}.
\end{align*}
\end{lemma}
\begin{proof} It is well known that \cite{IR}
\begin{gather*}
\sum_{r=0}^{n-1}r^m=\frac{B_{m+1}(n)-B_{m+1}}{m+1},\\
B_m(x+y)=\sum_{r=0}^m\binom{m}rB_{m-r}(x)y^r
\end{gather*}
and
$$
B_m\left(\frac12\right)=(2^{1-m}-1)B_m.
$$
For each odd integer $n>1$, we know $B_n=0$, hence $B_m\left(\frac12\right)=0$. By Euler's theorem, we see that
\begin{align*}
&\sum_{k=1}^{\frac{p-1}2}\frac1k\equiv\sum_{k=1}^{\frac{p-1}2}k^{\varphi(p^4)-1}=\frac{B_{\varphi(p^4)}(\frac{p+1}2)-B_{\varphi(p^4)}}{\varphi(p^4)}\\
=&\frac{B_{\varphi(p^4)}(\frac{p+1}2)-B_{\varphi(p^4)}\left(\frac12\right)}{\varphi(p^4)}+\frac{B_{\varphi(p^4)}\left(\frac{1}2\right)-B_{\varphi(p^4)}}{\varphi(p^4)}\\
=&\sum_{k=1}^{\varphi(p^4)}\frac1k\binom{\varphi(p^4)-1}{k-1}B_{\varphi(p^4)-k}\left(\frac12\right)\left(\frac{p}2\right)^k+\frac{B_{\varphi(p^4)}\left(\frac{1}2\right)-B_{\varphi(p^4)}}{\varphi(p^4)}\\
\equiv&\frac{p^2}8(\varphi(p^4)-1)B_{\varphi(p^4)-2}\left(\frac12\right)+\frac{B_{\varphi(p^4)}\left(\frac{1}2\right)-B_{\varphi(p^4)}}{\varphi(p^4)}\\
\equiv&-\frac78p^2B_{\varphi(p^4)-2}-2\frac{2^{\varphi(p^4)}-1}{\varphi(p^4)}B_{\varphi(p^4)}\pmod{p^4}.
\end{align*}
In view of \cite[(1.1)]{s2000}, we have
\begin{align*}
-\frac12B_{\varphi(p^4)-2}&\equiv\frac{B_{\varphi(p^4)-2}}{\varphi(p^4)-2}\equiv(p^3-1)\frac{B_{2p-4}}{2p-4}-(p^3-2)\left(1-p^{p-4}\right)\frac{B_{p-3}}{p-3}\\
&\equiv-\frac{B_{2p-4}}{2p-4}+2\frac{B_{p-3}}{p-3}=2X\pmod{p^2}.
\end{align*}
Since $pB_{\varphi(p^4)}\equiv p-1\pmod{p^4}$ by \cite[Corollary 4.1]{s1997}, we can deduce that
\begin{align*}
&\frac{2^{\varphi(p^4)}-1}{\varphi(p^4)}B_{\varphi(p^4)}=\frac{2^{\varphi(p^4)}-1}{p^4}\frac{pB_{\varphi(p^4)}}{p-1}\equiv\frac{(1+pq_p(2))^{p^3}-1}{p^4}\\
\equiv&\frac1{p^4}\left(\binom{p^3}1pq_p(2)+\binom{p^3}2p^2q^2_p(2)+\binom{p^3}3p^3q^3_p(2)+\binom{p^3}4p^4q^4_p(2)\right)\\
\equiv&q_p(2)-\frac{1}2pq^2_p(2)+\frac13p^2q^3_p(2)-\frac14p^3q^4_p(2)\pmod{p^4}.
\end{align*}
Thus,
\begin{align*}
H_{\frac{p-1}2}\equiv-2q_p(2)+pq^2_p(2)-\frac23p^2q^3_p(2)+\frac12p^3q^4_p(2)+\frac72p^2X\pmod{p^4}.
\end{align*}
Therefore the proof of Lemma \ref{Lem3.1} is complete.
\end{proof}
\noindent {\it Proof of Theorem \ref{Th1.2}}. We can find and prove the following identity by \texttt{Sigma},
$$
\sum_{k=1}^n\frac{\binom{n}k(-4)^k}{k^2\binom{2k}k}=-2\sum_{k=1}^n\frac1k\sum_{j=1}^k\frac1{2j-1}.
$$
It is easy to check that
\begin{align*}
&-2\sum_{k=1}^{\frac{p-1}2}\frac1k\sum_{j=1}^k\frac1{2j-1}=H\left(1,1;\frac{p-1}2\right)+H\left(2;\frac{p-1}2\right)\\
&-2H(1,1;p-1)-2H(2;p-1)-2H(1,-1;p-1)-2H(-2;p-1).
\end{align*}
By \cite[Lemma 4.2]{S11a}, we have
\begin{align*}
\sum_{k=1}^{\frac{p-1}2}\frac{\binom{\frac{p-1}2}k(-4)^k}{k^2\binom{2k}k}&\equiv\sum_{k=1}^{\frac{p-1}2}\frac1{k^2}\bigg(1-p\sum_{j=1}^k\frac1{2j-1}+\frac{p^2}2\left(\sum_{j=1}^k\frac1{2j-1}\right)^2\\
&\ \ \ \ \ \ \ \ \ \  -\frac{p^2}2\sum_{j=1}^k\frac1{(2j-1)^2}\bigg)\pmod{p^3}.
\end{align*}
This, with the identity $H(1,1;n)=\frac12(H_{n}^2-H(2;n))$ yields that
\begin{align*}
&\frac{p^2}2\left(\sum_{k=1}^{\frac{p-1}2}\frac1{k^2}\left(\sum_{j=1}^k\frac1{2j-1}\right)^2+\sum_{k=1}^{\frac{p-1}2}\frac1{k^3}\sum_{j=1}^k\frac1{2j-1}+\sum_{k=1}^{\frac{p-1}2}\frac1{k^2}\sum_{j=1}^k\frac1{(2j-1)^2}\right)\\
&\equiv\frac12H_{\frac{p-1}2}^2-\frac12H\left(2;\frac{p-1}2\right)-H_{p-1}^2-H(2;p-1)-2H(-2;p-1)\\
&\ \ \ \ -2H(1,-1;p-1)+p\sum_{k=1}^{\frac{p-1}2}\frac1{k^2}\sum_{j=1}^k\frac1{2j-1}+p^2\sum_{k=1}^{\frac{p-1}2}\frac1{k^2}\sum_{j=1}^k\frac1{(2j-1)^2}\\
&\ \ \ \ +\frac{p^2}2\sum_{k=1}^{\frac{p-1}2}\frac1{k^3}\sum_{j=1}^k\frac1{2j-1}\pmod{p^3}.
\end{align*}
By (viii), we have
\begin{align*}
H(-2;p-1)&=\sum_{k=1}^{p-1}\frac{(-1)^k+1}{k^2}-H(2;p-1)\\
&=\frac12H\left(2;\frac{p-1}2\right)-H(2;p-1)\equiv-3pX\pmod{p^3}.
\end{align*}
These, with (i), (viii), (\ref{u2})-(\ref{u3}), Lemma \ref{Lem2.5} and \cite[(59),(62)]{TZ} yield that
\begin{align}\label{zhongyao2}
&\frac{p^2}2\left(\sum_{k=1}^{\frac{p-1}2}\frac1{k^2}\left(\sum_{j=1}^k\frac1{2j-1}\right)^2+\sum_{k=1}^{\frac{p-1}2}\frac1{k^3}\sum_{j=1}^k\frac1{2j-1}+\sum_{k=1}^{\frac{p-1}2}\frac1{k^2}\sum_{j=1}^k\frac1{(2j-1)^2}\right)\notag\\
&\equiv\frac12H_{\frac{p-1}2}^2+\frac{13}2pX-2H(1,-1;p-1)+p^2q_p(2)B_{p-3}-\frac14p^2h_{31}\notag\\
&\ \ \ \ \ \ \ \ \ \ \ \ \ \ \ -2p^2H(1,-3;p-1)\pmod{p^3}.
\end{align}
It is easy to see that
\begin{align*}
&-2q_p(2)=\sum_{k=1}^{p-1}\frac{(-1)^k}k\binom{p-1}{k-1}\equiv\sum_{k=1}^{p-1}\frac{(-1)^k}k\bigg(1-pH(1;k-1)\\
&\ \ \ \ \ \ \ \ \ \ \ \ \ \ \ \ \ \ \ \ \ +p^2H(1,1;k-1)-p^3H(1,1,1;k-1)\bigg)\\
&=H_{\frac{p-1}2}-H_{p-1}-pH(1,-1;p-1)+p^2H(1,1,-1;p-1)\\
&\ \ \ \ \ \ -p^3H(1,1,1,-1;p-1)\pmod{p^4}.
\end{align*}
In view of \cite[(87) and Remark 7.7]{TZ}, we have
\begin{align*}
H(1,1,-1;p-1)\equiv\sum_{k=1}^{p-1}\frac{2^k}{k^3}&+pH(1,1,1,-1;p-1)-\frac{7}{12}pq_p(2)B_{p-3}\\
&+pH(1,-3;p-1)\pmod{p^2}.
\end{align*}
These, with (viii) yield that
\begin{align*}
H(1,-1;p-1)\equiv&\frac{H_{\frac{p-1}2}+2q_p(2)}{p}+p\sum_{k=1}^{p-1}\frac{2^k}{k^3}-\frac{7}{12}p^2q_p(2)B_{p-3}\\
&+p^2H(1,-3;p-1)-2pX\pmod{p^3}.
\end{align*}
Thus, in view of (\ref{zhongyao2}) and Lemma \ref{Lem2.6}, we have
\begin{align*}
\sum_{k=1}^{p-1}\frac{2^k}{k^3}\equiv&\frac1{4p}H_{\frac{p-1}2}^2+\frac{21}4X-\frac{H_{\frac{p-1}2}+2q_p(2)}{p^2}+\frac{13}{12}pq_p(2)B_{p-3}\\
&-2pH(1,-3;p-1)-\frac18ph_{31}\pmod{p^2}.
\end{align*}
Therefore, by Lemmas \ref{Lem2.5}-\ref{Lemzhongyao} and \ref{Lem3.1}, we immediately obtain the desired result 
$$
\sum_{k=1}^{p-1}\frac{2^k}{k^3}\equiv-\frac13q^3_p(2)+\frac74X+\frac5{12}pq^4_p(2)+\frac{7p}6q_p(2)B_{p-3}-\frac{3p}8\sum_{k=1}^{\frac{p-1}2}\frac{H_k}{k^3}\pmod{p^2}.
$$
Now the proof of Theorem \ref{Th1.2} is complete.\qed

\end{document}